\documentclass[10pt,a4paper,reqno]{amsart} 

\usepackage{amssymb}
\usepackage{amsthm}
\usepackage{amsmath}
\usepackage{mathrsfs}
\usepackage{amsbsy}
\usepackage[all]{xy}
\usepackage{bm}
\usepackage{tikz}
\usepackage{array}
\usepackage{float}
\usepackage{enumerate}
\usepackage{xcolor}
\usepackage{hhline}
\allowdisplaybreaks
\usepackage{cite}
\usepackage{tabularx,booktabs}

\usepackage[english]{babel}
\usepackage[utf8]{inputenc}
\usepackage{latexsym}
\usepackage{verbatim}
\usepackage{rotating}
\usepackage[colorlinks=true,citecolor=black,linkcolor=black,urlcolor=blue]{hyperref}

\usepackage{color}
\usepackage{listings}
\usepackage{caption}
\usepackage{amsaddr}

\newtheorem{thm}{Theorem}[section] 

\newtheorem{lemma}[thm]{Lemma} 

\newtheorem{prop}[thm]{Proposition} 

\theoremstyle{definition}

\newtheorem{remarkx}[thm]{Remark}
\newenvironment{remark}
  {\pushQED{\qed}\remarkx}
  {\popQED\endremarkx}

\newtheorem{exmp}[thm]{Example} 

\catcode`\@=11
\def\section{\@startsection{section}{1}%
 \z@{.7\linespacing\@plus\linespacing}{.5\linespacing}%
 {\normalfont\bfseries\scshape\centering}}

\def\subsection{\@startsection{subsection}{2}%
  \z@{.5\linespacing\@plus\linespacing}{.5\linespacing}%
  {\normalfont\bfseries\scshape}}

\def\subsubsection{\@startsection{subsubsection}{3}%
 \z@{.5\linespacing\@plus\linespacing}{-.5em}
 {\normalfont\bfseries}}
\catcode`\@=12

\DeclareMathOperator{\SW}{SW}
\DeclareMathOperator{\tl}{tl}
\DeclareMathOperator{\leg}{leg}
\DeclareMathOperator{\id}{id}

\newcounter{nalg} 
\DeclareCaptionLabelFormat{algocaption}{Algorithm \thenalg} 

\lstnewenvironment{algorithm}[1][] 
{   
    \refstepcounter{nalg} 
    \captionsetup{labelformat=algocaption,labelsep=colon} 
    \lstset{ 
        mathescape=true,
        frame=tB,
        keywordstyle=\color{black}\bfseries\em,
        keywords={,input, output, return, datatype, function, in, if, else, foreach, while, begin,  for, from, to, do,make, on, between, and, decreasing, metadata, add,} 
        xleftmargin=.04\textwidth,
        #1 
    }
}
{}

\begin{document}

\title[]{Asymptotics of 3-Stack-Sortable Permutations}
\subjclass[2010]{}

\author[]{Colin Defant}
\address[]{Department of Mathematics, Princeton University, Princeton, NJ 08544}
\email{cdefant@princeton.edu}
\author[]{Andrew Elvey Price}
\address[]{CNRS, Institut Denis Poisson, Universit\'e de Tours, France\\
}
\email{andrewelveyprice@gmail.com}
\author[]{Anthony J. Guttmann}
\address[]{School of Mathematics and Statistics, The University of Melbourne, Victoria 3010, Australia}
\email{guttmann@unimelb.edu.au}
\thanks{The first author is supported by an NSF Graduate Research Fellowship (grant DGE--1656466) and a Fannie and John Hertz Foundation Fellowship. The second author was supported by the European Research Council (ERC) in the European Union’s
Horizon 2020 research and innovation programme, under the Grant Agreement No. 759702.}
\maketitle

\begin{abstract}
We derive a simple functional equation with two catalytic variables characterising the generating function of 3-stack-sortable permutations. Using this functional equation, we extend the 174-term series to 1000 terms. From this series, we conjecture that the generating function behaves as 
$$W(t) \sim C_0(1-\mu_3 t)^\alpha \cdot \log^\beta(1-\mu_3 t), $$
so that $$[t^n]W(t)=w_n \sim \frac{c_0\mu_3^n}{  n^{(\alpha+1)}\cdot \log^\lambda{n}} ,$$
where $\mu_3 = 9.69963634535(30),$ $\alpha = 2.0 \pm 0.25.$ If $\alpha = 2$ exactly, then $\lambda = -\beta+1$, and we estimate $\beta \approx -2,$ but with a wide uncertainty of $\pm 1.$  If $\alpha$ is not an integer, then $\lambda=-\beta$, but we cannot give a useful estimate of $\beta$. The growth constant estimate (just) contradicts a conjecture of the first author that $$9.702 < \mu_3 \le 9.704.$$ We also prove a new rigorous lower bound of $\mu_3\geq 9.4854$, allowing us to disprove a conjecture of B\'ona. 

We then further extend the series using differential-approximants to obtain approximate coefficients $O(t^{2000}),$ expected to be accurate to $20$ significant digits, and use the approximate coefficients to provide additional evidence supporting the results obtained from the exact coefficients.
\end{abstract}
\makeatletter
\@setabstract
\makeatother

\section{Introduction}
A permutation is said to be {\em stack-sortable} if, when it is passed through a stack, the result is the increasing permutation. We describe the stack-sorting map below. Knuth \cite{K68} showed that a permutation is stack-sortable if and only if it avoids the pattern 231; such permutations are counted by the Catalan numbers $C_n = \binom {2n}{n}/(n+1) \sim 4^n/\sqrt{n^3\pi}$. 

A permutation is called {\em $k$-stack-sortable} if iteratively applying the stack-sorting map to it $k$ times results in the increasing permutation. Let $\mathcal W_k(n)$ be the set of $k$-stack-sortable permutations in $S_n$. West \cite{W90} conjectured, and Zeilberger \cite{Z92} subsequently proved, that $$|\mathcal W_2(n)| = \frac{2}{(n+1)(2n+1)}\binom{3n}{n} \sim \left (\frac{27}{4} \right )^n\frac{1}{2} \sqrt{\frac{3}{\pi n^5}}.$$ There is no known formula for $|\mathcal W_k(n)|$ when $k\geq 3$ is fixed. We are primarily interested in the numbers $|\mathcal W_3(n)|$, which we will denote by $w_n$.

Recently, the first author \cite{D20} developed a polynomial-time algorithm to generate the numbers $w_n$, thus providing a dramatic extension of the pre-existing 13-term sequence; the sequence is given in the OEIS as A134664 \cite{OEIS}. We start by reinterpreting these arguments to derive a functional equation that characterises this generating function. Using this functional equation, we then write an efficient algorithm with which we compute the numbers $w_{n}$ for $n\leq 1000$. This type of algorithm could apply to a wide variety of combinatorial functional equations, yet to our knowledge it is new to enumerative combinatorics. 

In the next section, we describe the derivation of the functional equation. We then describe the algorithm in Section~\ref{Sec:Series}. The remaining sections are devoted to a careful analysis of the 1000-term series that we use to conjecture the asymptotic behaviour. We first use standard methods of series analysis, looking at the behaviour of the ratios of successive coefficients and then analysing the series by the method of differential-approximants \cite{G89}. We then use the differential-approximants to (approximately) extend the series from 1000 terms to 2000 terms, using a method of series extension developed by the third author \cite{G16}, where the newly-obtained coefficients are expected to be accurate to at least 20 significant digits. This is more than sufficient for ratio-based analysis methods, so we then re-analyse the extended series.

\subsection{Background and notation}
Throughout this article, a \emph{permutation}
is an ordering of a finite set of positive integers. We let $S_n$ denote the set of permutations of the set $[n]:=\{1,\ldots,n\}$. Let $\id_n=123\cdots n$ be the identity permutation in $S_n$. The \emph{standardisation} of a permutation $\pi=\pi_1\cdots\pi_n$ is the permutation in $S_n$ obtained by replacing the $i$th-smallest entry in $\pi$ by $i$ for all $i$. For example, the standardisation of $5971$ is $2431$. We say two permutations have the \emph{same relative order} if their standardisations are equal. We say a permutation $\pi$ \emph{contains} a permutation $\tau$ as a pattern if there is a (not necessarily consecutive) subsequence of $\pi$ that has the same relative order as $\tau$; otherwise, we say $\pi$ \emph{avoids} $\tau$. A \emph{descent} of $\pi$ is an index $i\in[n-1]$ such that $\pi_i>\pi_{i+1}$. A \emph{peak} of $\pi$ is an index $i\in\{2,\ldots,n-1\}$ such that $\pi_{i-1}<\pi_i>\pi_{i+1}$. 

The stack-sorting map is a function $s$ that sends permutations to permutations; there is a simple recursive definition of this map. First, $s$ sends the empty permutation to itself. Now suppose $\pi$ is a nonempty permutation with largest entry $m$, and write $\pi=LmR$. Then $s(\pi)=s(L)s(R)m$. For example, \[s(326451)=s(32)\,s(451)\,6=2\,s(3)\,s(4)\,s(1)\,56=234156.\] A simple consequence of this definition, which we will use tacitly in what follows, is that $|s^{-1}(\pi)|=|s^{-1}(\pi')|$ whenever $\pi$ and $\pi'$ have the same relative order. 

We say a permutation $\pi$ is \emph{$k$-stack-sortable} if $s^k(\pi)$ is increasing, where $s^k$ denotes the $k$-fold iterate of $s$. Let $\mathcal W_k(n)$ denote the set of $k$-stack-sortable permutations in $S_n$. Let $\mathcal W_k=\bigcup_{n\geq 0}\mathcal W_k(n)$.

The following theorem is essentially due to Knuth; it is the theorem that initiated the study of permutation patterns. 

\begin{thm}[\!\!\cite{K68}]\label{Thm:Knuth}
A permutation is $1$-stack-sortable if and only if it avoids the pattern $231$. The number of $1$-stack-sortable permutations in $S_n$ is the Catalan number $C_n=\frac{1}{n+1}{2n\choose n}$. 
\end{thm} 

The preceding theorem tells us that if $X$ is any finite set of $n$ positive integers and $\pi$ is the permutation obtained by listing the elements of $X$ in increasing order, then $|s^{-1}(\pi)|=C_n$. 

In his dissertation, West \cite{W90} characterised the $2$-stack-sortable permutations as follows. 

\begin{thm}[\!\!\cite{W90}]\label{Thm:West}
A permutation is $2$-stack-sortable if and only if it avoids the pattern $2341$ and also avoids any $3241$ pattern that is not part of a $35241$ pattern. 
\end{thm}

As an example of West's characterisation, consider the permutation $\pi=416352$. Note that $\pi$ avoids $2341$. On the other hand, $\pi$ does not avoid $3241$ because it contains the subsequence $4352$. However, this subsequence is part of the occurrence $46352$ of the pattern $35241$, so it does not prevent $416352$ from being $2$-stack-sortable. There are no other occurrences of $3241$ in $\pi$, so $\pi$ is $2$-stack-sortable. 

In \cite{Ulfarsson}, \'Ulfarsson gave an analogous characterisation of $3$-stack-sortable permutations. However, his characterisation is too unwieldy to be useful for enumerative purposes. In order to obtain a functional equation for the generating function of $3$-stack-sortable permutations, we will instead make use of the Decomposition Lemma from \cite{D20}.

\section{The Functional Equation}\label{functional_equation}
In this section we derive a functional equation characterising the generating function of 3-stack-sortable permutations, as summarised by the following theorem. Recall that we let $w_n=|\mathcal W_3(n)|$ denote the number of $3$-stack-sortable permutations in $S_n$.
\begin{thm}\label{thm:Characterisation_functional_equation}
There exists a unique series $Q(x,a)\equiv Q(t,x,a)$ belonging to $\mathbb{Z}[a,x][[t]]$ that satisfies the equation
\begin{equation*}Q(x,a)=t(x+1)^2(1+a)^2+t(1+x)\frac{Q(x,a)-Q(x,0)}{a}\left((1+a)^2+a\frac{Q(x,a)-Q(0,a)}{x}\right)\end{equation*} \[+t(1+x)aQ(x,a).\]
The generating function 
\[W(t)\equiv\sum_{n=0}^{\infty}w_{n}t^{n}\] counting 3-stack-sortable permutations is given by $W(t)=Q(t,0,0)$.
\end{thm}
We have been unable to solve this functional equation exactly. However, we will use it in subsequent sections to compute $1000$ terms of the series $W(t)$, which we will then analyse empirically.

In order to derive the functional equation, we will represent a permutation $\pi=\pi_1\cdots\pi_n$ via its \emph{plot}, which is the diagram showing the points $(i,\pi_i)$ for all $i\in[n]$. A \emph{hook} of $\pi$ is a sideways L shape that connects two points $(i,\pi_i)$ and $(j,\pi_j)$ in the plot of $\pi$ such that $i<j$ and $\pi_i<\pi_j$. The point $(i,\pi_i)$ is the \emph{southwest endpoint} of the hook, and $(j,\pi_j)$ is the \emph{northeast endpoint} of the hook. Let $\SW_i(\pi)$ be the set of hooks of $\pi$ with southwest endpoint $(i,\pi_i)$. For example, if $\pi=426315789$, then the hook shown in Figure~\ref{Fig4} is in $\SW_3(\pi)$ because its southwest endpoint is $(3,6)$; its northeast endpoint is $(8,8)$. 

\begin{figure}[h]
  \begin{center}{\includegraphics[width=0.24\textwidth]{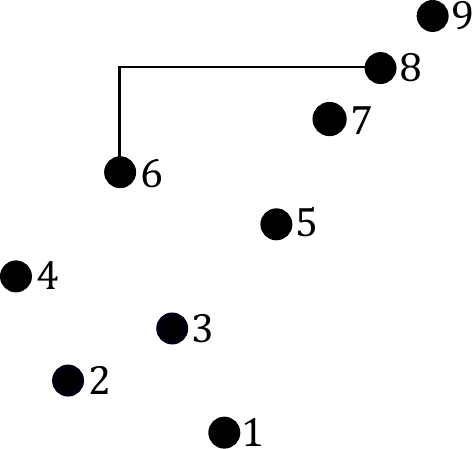}}
  \end{center}
  \caption{The plot of $426315789$ along with a single hook.}\label{Fig4}
\end{figure}

The \emph{tail length} of a permutation $\pi=\pi_1\cdots\pi_n\in S_n$, denoted $\tl(\pi)$, is the largest integer $\ell\in\{0,\ldots,n\}$`such that $\pi_i=i$ for all $i\in\{n-\ell+1,\ldots,n\}$. For example, $\tl(23145) = 2$, $\tl(23154) = 0$, and $\tl(12345) = 5$. The \emph{tail} of $\pi$ is the sequence of points $(n-\tl(\pi)+1,n-\tl(\pi)+1),\ldots,(n,n)$ in the plot of $\pi$. For example, the tail of $426315789$ is $(7,7),(8,8),(9,9)$. We say a descent $d$ of $\pi$ is \emph{tail-bound} if every hook in $\SW_d(\pi)$ has its northeast endpoint in the tail of $\pi$. The descents of $426315789$ are $1$, $3$, and $4$, but the only tail-bound descent is $3$. For $\pi\in S_n\setminus\{\id_n\}$, let $c(\pi)$ denote the index such that $\pi_{c(\pi)}=n-\tl(\pi)$; note that $c(\pi)$ is a tail-bound descent of $\pi$. 

We now describe a way to decompose permutations, which will lead to a functional equation for the generating function of 3-stack-sortable permutations. Let $H$ be a hook of $\pi$ with southwest endpoint $(i,\pi_i)$ and northeast endpoint $(j,\pi_j)$. Let $\pi_U^H=\pi_1\cdots\pi_i\pi_{j+1}\cdots\pi_n$ and 
$\pi_S^H=\pi_{i+1}\cdots\pi_{j-1}$. The permutations $\pi_U^H$ and $\pi_S^H$ are called the \emph{$H$-unsheltered subpermutation of $\pi$} and the \emph{$H$-sheltered subpermutation of $\pi$} respectively. For instance, if $\pi=426315789$ and $H$ is the hook shown in Figure~\ref{Fig4}, then $\pi_U^H=4269$ and $\pi_S^H=3157$. We will deal exclusively with situations in which $i$ is a tail-bound descent of $\pi$, in which case the plot of $\pi_S^H$ lies entirely beneath the hook $H$ (it is ``sheltered" by $H$). 

\begin{lemma}[Decomposition Lemma \cite{D20}]
If $d$ is a tail-bound descent of a permutation $\pi$, then \[|s^{-1}(\pi)|=\sum_{H\in\SW_d(\pi)}|s^{-1}(\pi_U^H)|\cdot|s^{-1}(\pi_S^H)|.\]
\end{lemma}

Theorem~\ref{Thm:West} gives us structural information about $2$-stack-sortable permutations. Since $\mathcal W_3=s^{-1}(\mathcal W_2)$, we can combine this structural information with the Decomposition Lemma in order to gain enumerative information about $3$-stack-sortable permutations. In doing so, we must keep track of the tail length statistic, as well as one additional statistic that we now describe.

For $a\in\{0,\ldots,n\}$, we say the open interval $(a,a+1)$ is a \emph{legal space for $\pi$} if there do not exist indices $i_1<i_2<i_3$ such that $\pi_{i_3}\leq a<\pi_{i_1}<\pi_{i_2}$. Let $\leg(\pi)$ be the number of legal spaces of $\pi$. For example, if $\pi\in S_n$, then $\leg(\pi)=n+1$ if and only if $\pi$ avoids $231$. The legal spaces of $145326$ are $(0,1),(1,2),(4,5),(5,6),(6,7)$, so $\leg(145326)=5$. Imagine taking the plot of a permutation $\pi$ and adding a new point to the left of all other points. One can think of the legal spaces of $\pi$ as the vertical positions where the new point can be inserted so as to not form a new $2341$ pattern. This is relevant for us because, as we know from Theorem~\ref{Thm:West}, every $2$-stack-sortable permutation avoids $2341$. 

We will prove a new functional equation for the generating function 
\[J(t,u,v)=\sum_{n\geq 1}\sum_{\pi\in\mathcal W_2(n)}|s^{-1}(\pi)|t^nu^{\leg(\pi)-1}v^{\tl(\pi)},\] allowing us to characterise the series $J(t,1,1)=\sum_{n\geq 1}w_nt^n$ (where $w_n=|\mathcal W_3(n)|$). Transforming this functional equation to that of a related series will yield Theorem~\ref{thm:Characterisation_functional_equation}.

One contribution to the generating function $J(t,u,v)$ comes from the preimages of the identity permutations: 
\begin{equation}\label{Eq33}
\sum_{n\geq 1}|s^{-1}(\id_n)|t^nu^{\leg(\id_n)-1}v^{\tl(\id_n)}=C(tuv)-1,
\end{equation}
where $C(x)=\sum_{n\geq 0}C_nx^n=\dfrac{1-\sqrt{1-4x}}{2x}$ is the generating function of the Catalan numbers. Indeed, this follows immediately from Theorem~\ref{Thm:Knuth}. 

Now suppose we want to construct a nonempty permutation $\sigma\in\mathcal W_3$ such that $s(\sigma)=\pi$ is not an identity permutation. The Decomposition Lemma tells us that the number of ways to do this is equal to the number of ways to choose $\pi\in\mathcal W_2$, $H\in\SW_{c(\pi)}(\pi)$, $\mu\in s^{-1}(\pi_U^H)$, and $\lambda\in s^{-1}(\pi_S^H)$. In order to choose $\pi$, we first choose the standardisations $\zeta_U$ and $\zeta_S$ of $\pi_U^H$ and $\pi_S^H$ respectively. We then combine them along with the additional hook $H$ to form $\pi$. When combining them, we need to choose how many points from each of $\pi_U^H$ and $\pi_S^H$ end up in the tail of $\pi$. We also need to choose the relative heights of the points of $\pi_U^H$ compared to the points of $\pi_S^H$. Because the resulting permutation $\pi$ is supposed to be $2$-stack-sortable, one can use Theorem~\ref{Thm:West} to check that $\zeta_U$ and $\zeta_S$ must be in $\mathcal W_2$. Furthermore, we cannot have more than one point in $\pi_U^H$ outside of the tail of $\pi$ that lies above a point in $\pi_S^H$. In other words, at most one point to the left of $(c(\pi),\pi_{c(\pi)})$ can lie above a point to the right of $(c(\pi),\pi_{c(\pi)})$. If such a point exists, it must be placed in a position corresponding to one of the legal spaces of $\pi_S^H$. Using Theorem~\ref{Thm:West}, one can verify that the permutation $\pi$ produced in this fashion is necessarily $2$-stack-sortable. 

\begin{exmp}\label{Exam5}
Suppose we choose the standardisations of $\pi_U^H$ and $\pi_S^H$ to be $\zeta_U=24315678$ and $\zeta_S=315246$ respectively. Furthermore, suppose we choose to put $3$ of the points of $\pi_U^H$ and $1$ of the points of $\pi_S^H$ into the tail of $\pi$. These choices almost determine the $2$-stack-sortable permutation $\pi$. Indeed, $\pi$ must be of the shape shown in Figure~\ref{Fig15}, where the open dashed circles represent the possible places where we can put the largest point to the left of $(c(\pi),\pi_{c(\pi)})$. These open dashed circles correspond to the legal spaces of $\pi_S^H$ lying below the point $(c(\pi),\pi_{c(\pi)})$.

\begin{figure}[h]
\begin{center}
\includegraphics[width=.43\linewidth]{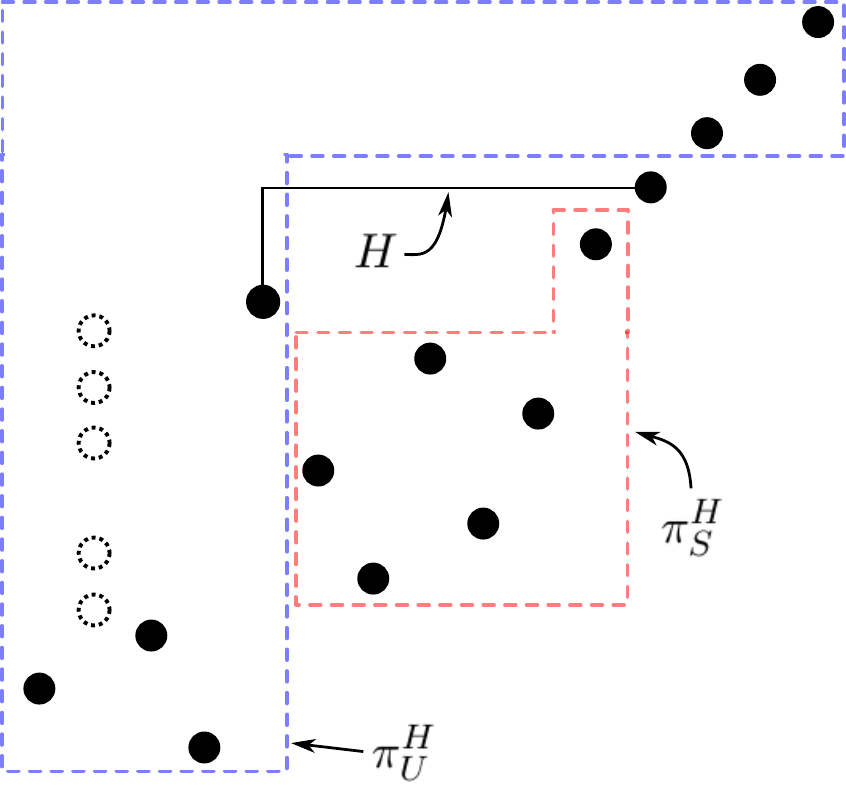}
\caption{Combining $\pi_U^H$ and $\pi_S^H$ to form $\pi$.}
\label{Fig15} \qedhere
\end{center}  
\end{figure}
\end{exmp}

\begin{thm}\label{Thm26}
We have \[J(t,u,v)=(C(tuv)-1)(1+tuJ(t,u,1))+\frac{tuv}{1-u}\left(\frac{J(t,u,1)-J(t,u,v)}{1-v}-\frac{C(tuv)-1}{v}\right)\] \[\cdot\left(\frac{J(t,1,1)-uvJ(t,1,uv)}{1-uv}-u\frac{J(t,u,1)-vJ(t,u,v)}{1-v}\right).\]
\end{thm}

\begin{proof}
As mentioned above, the contribution to $J(t,u,v)$ coming from preimages of identity permutations is $C(tuv)-1$. Let us now go through the procedure described above in order to find a preimage of a nonidentity permutation in $\mathcal W_2$. We first consider the case in which there are no points to the left of $(c(\pi),\pi_{c(\pi)})$ (i.e., $c(\pi)=1$). Equivalently, the standardisation $\zeta_U$ of $\pi_U^H$ is an identity permutation of some length $m\geq 1$, and $m-1$ of the points of $\pi_U^H$ are in the tail of $\pi$. We can choose any permutation in $\mathcal W_2$ to be the standardisation $\zeta_S$ of $\pi_S^H$. If $\zeta_S$ is not an identity permutation, then the number of points of $\pi_S^H$ that we place in the tail of $\pi$ can be any element of $\{0,\ldots,\tl(\zeta_S)\}$. If $\zeta_S$ is an identity permutation of some length $m'$, then the number of points of $\pi_S^H$ that we place in the tail of $\pi$ can be any element of $\{0,\ldots,m'-1\}$ (choosing all $m'$ points to go in the tail of $\pi$ would result in $\pi$ being an identity permutation). There are then $C_m$ ways to choose $\mu\in s^{-1}(\pi_U^H)$ and $|s^{-1}(\zeta_S)|$ ways to choose $\lambda\in s^{-1}(\pi_S^H)$. The total contribution to $J(t,u,v)$ coming from this case is \[tuv\frac{C(tuv)-1}{v}\left(\sum_{m'\geq 1}\sum_{\zeta_S\in\mathcal W_2(m')}|s^{-1}(\zeta_S)|t^{m'}u^{\leg(\zeta_S)-1}(1+v+\cdots+v^{\tl(\zeta_S)})-(C(tuv)-1)\right)\] 
\begin{equation}\label{Eq34}
=tu(C(tuv)-1)\left(\frac{J(t,u,1)-vJ(t,u,v)}{1-v}-C(tuv)+1\right),
\end{equation}
where the initial factor $tuv$ accounts for the northeast endpoint of $H$.

We now consider the case in which there is at least one point in the plot of $\pi$ to the left of $(c(\pi),\pi_{c(\pi)})$ (i.e., $c(\pi)\geq 2$). We choose $\zeta_S$ along with the number $\ell$ of points of $\pi_S^H$ that go into the tail of $\pi$. There are then $\leg(\pi_S^H)-\ell$ choices for the position of the highest point to the left of $(c(\pi),\pi_{c(\pi)})$. We also need to choose the standardisation $\zeta_U$ of $\pi_U^H$ and the number $j$ of points of $\pi_U^H$ appearing in the tail of $\pi$. Finally, there are $|s^{-1}(\zeta_U)|$ ways to choose $\mu\in s^{-1}(\pi_U^H)$ and $|s^{-1}(\zeta_S)|$ ways to choose $\lambda\in s^{-1}(\pi_S^H)$. The total contribution to the generating function from this case is \[tuv\left(\sum_{m\geq 1}\sum_{\zeta_U\in\mathcal W_2(m)}|s^{-1}(\zeta_U)|t^mu^{\leg(\zeta_U)-1}\sum_{j=0}^{\tl(\zeta_U)-1}v^j-\frac{C(tuv)-1}{v}\right)\] 
\begin{equation}\label{Eq35}
\cdot\left(\sum_{m'\geq 1}\sum_{\zeta_S\in\mathcal W_2(m')}|s^{-1}(\zeta_S)|t^{m'}\sum_{\ell=0}^{\tl(\zeta_S)}v^\ell\sum_{r=\ell}^{\leg(\zeta_S)-1}u^r-(C(tuv)-1)\right).
\end{equation} The term $-\dfrac{C(tuv)-1}{v}$ appears so that we do not count the contribution coming from \eqref{Eq34} again. The term $-(C(tuv)-1)$ occurs because if $\zeta_S=\id_{m'}$, then we cannot choose all of the points of $\pi_S^H$ to appear in the tail of $\pi$. The last displayed expression simplifies to 
\[
tuv\left(\frac{J(t,u,1)-J(t,u,v)}{1-v}-\frac{C(tuv)-1}{v}\right)\] 
\begin{equation}\label{Eq36}
\cdot\left(\frac{J(t,1,1)-uvJ(t,1,uv)}{(1-u)(1-uv)}-u\frac{J(t,u,1)-vJ(t,u,v)}{(1-u)(1-v)}-(C(tuv)-1)\right).
\end{equation}

Finally, $J(t,u,v)$ is the sum of the expressions in \eqref{Eq33}, \eqref{Eq34}, and \eqref{Eq36}. This sum simplifies to the expression stated in the theorem. 
\end{proof}

\begin{remark}
Theorem~\ref{Thm26} provides a functional equation corresponding to a recurrence relation for the numbers $w_n=|\mathcal W_3(n)|$ that was obtained in \cite{D20}. In fact, that article gave a more general recurrence that counts $3$-stack-sortable permutations according to their descent and peak statistics. These recurrences were vastly generalised in \cite[Theorem 9.8]{DefantTroupes}, where they were phrased in terms of postorder readings and special sets of colored binary plane trees called troupes.
We wish to remark that the proof of Theorem~\ref{Thm26} generalises immediately to give functional equations corresponding to these more general recurrences. For the sake of brevity, we omit the details of this discussion. However, there are a couple of special cases that are worth mentioning. 

A permutation of length $n$ is called \emph{alternating} if its descents are precisely the even elements of $[n-1]$. Let $\text{ALT}$ denote the set of alternating permutations of odd length. If we replace each occurrence of the generating function $C(tuv)=\dfrac{1-\sqrt{1-4tuv}}{2tuv}$ in Theorem~\ref{Thm26} with \[\frac{1-\sqrt{1-4(tuv)^2}}{2tuv},\] then the resulting generating function $J(t,u,v)$ is such that \[J(t,1,1)=\sum_{n\geq 1}|\mathcal W_3(n)\cap\text{ALT}|t^n.\]

Let $\text{EDP}$ denote the set of permutations $\pi$ in which every descent is a peak. If we replace each occurrence of the generating function $C(tuv)=\dfrac{1-\sqrt{1-4tuv}}{2tuv}$ in Theorem~\ref{Thm26} with \[\frac{1-tuv-\sqrt{1-2tuv-3(tuv)^2}}{2(tuv)^2},\] then the resulting generating function $J(t,u,v)$ is such that \[J(t,1,1)=\sum_{n\geq 1}|\mathcal W_3(n)\cap\text{EDP}|t^n. \qedhere\]
\end{remark}

In the remainder of this section, we transform the series $J(t,u,v)$ into a new series $Q(t,x,a)$ that satisfies the simpler functional equation of Theorem \ref{thm:Characterisation_functional_equation}.

By the definition of $J$, we clearly have $J(t,u,v)\in\mathbb{Z}[u,v][[t]]$, and for each monomial $u^k v^{\ell} t^n$ appearing, we have $\ell\leq n,k$. As a consequence, the series
\[J_{1}(t,u,w)=J\left(t,u,\frac{w}{tu}\right)\]
lies in $\mathbb{Z}[u][[t,w]]$. Observe that the series
\[J_{2}(t,u,w)=u\frac{wJ_{1}(t,u,w)-tuJ_{1}(t,u,tu)}{w-tu}-u(C(w)-1)\]
must also lie in $\mathbb{Z}[u][[t,w]]$. Finally we define the series $Q(t,x,a)\in\mathbb{Z}[x][[t,a]]$ by
\[Q(t,x,a)=J_{2}\left(t,x+1,\frac{a}{(1+a)^2}\right).\]
This last transformation is convenient as the terms $C(w)$ that appear in the functional equation for $J_{2}(t,u,w)$ are now $C(\frac{a}{(1+a)^2})=a+1.$
The equation in terms of $Q(x,a)\equiv Q(t,x,a)$ is
\begin{equation}\label{eq:transformed_fe}Q(x,a)=t(x+1)^2(1+a)^2+t(1+x)\frac{Q(x,a)-Q(x,0)}{a}\left((1+a)^2+a\frac{Q(x,a)-Q(0,a)}{x}\right)\end{equation} \[+t(1+x)aQ(x,a),\]
and the generating function for 3-stack-sortable permutations is $W(t)=J(t,1,1)=J_{2}(t,1,0)=Q(t,0,0)$.

To see that \eqref{eq:transformed_fe} characterises the generating function $Q(t,x,a)$, we write
\[Q(t,x,a)=\sum_{n=0}^{\infty}t^{n}Q_{n}(x,a),\]
where each $Q_{n}(x,a)\in\mathbb{Z}[x][[a]]$. The equation \eqref{eq:transformed_fe} immediately implies that $Q_{0}(x,a)=0$, from which it follows that $Q_{1}(x,a)=(x+1)^2(a+1)^2$. We can then rewrite \eqref{eq:transformed_fe} as the following recurrence for $n\geq 2$:
\begin{align}Q_{n}(x,a)&=(1+x)(1+a)^2\frac{Q_{n-1}(x,a)-Q_{n-1}(x,0)}{a}+(1+x)aQ_{n-1}(x,a) \nonumber\\
&~~+\frac{t(1+x)}{x}\sum_{j=1}^{n-2}(Q_{j}(x,a)-Q_{j}(x,0))(Q_{n-j-1}(x,a)-Q_{n-j-1}(0,a)).\label{eq:recursiveQ}\end{align}
Clearly this uniquely defines the series $Q_{n}(x,a)$, completing the proof of Theorem \ref{thm:Characterisation_functional_equation}. Moreover, a simple induction shows that each series $Q_{n}(x,a)$ is in fact a polynomial of degree $n+1$ in each variable.

\begin{remark}
We note the functional equation \eqref{eq:transformed_fe} appears to be just outside the boundary of functional equations that can be systematically solved. In particular, this equation has two catalytic variables $a$ and $x$ and takes the form of a polynomial equation in $t$, $a$, $x$, $Q(x,a)$, $Q(x,0)$ and $Q(0,a)$. If either $Q(x,0)$ or $Q(0,a)$ did not appear there would only be one catalytic variable, in which case the series would be algebraic \cite{BM-J}. Moreover, increasingly systematic methods are being developed for functional equations with two catalytic variables where the polynomial is linear in $Q(x,a)$ (see for example \cite{B-BM,BM-M,B-BM-R,DHRS}), whereas \eqref{eq:transformed_fe} is quadratic in $Q(x,a)$. 
\end{remark}

\section{Series generation}\label{Sec:Series}

\subsection{Computing coefficients of the generating function}

We now describe numerous algorithms to compute the numbers $w_n=Q_{n}(0,0)$ for $n\leq N$, where $N$ is some positive integer. Each algorithm is based on the recurrence \eqref{eq:recursiveQ}.

The most direct technique would be to compute each polynomial $Q_{n}(x,a)$ directly from \eqref{eq:recursiveQ}. However, computing a product of $Q_{j}(x,a)-Q_{j}(x,0)$ and $Q_{n-j-1}(x,a)-Q_{n-j-1}(0,a)$ would take $\Theta(n^4)$ individual integer multiplications. Since there are $\Theta(n^2)$ such products to compute in total and the size of the integers is $\Theta(n)$, the algorithm would then have complexity $\Theta(N^8)$ (using naive multiplication). We note that this could be reduced to around $\Theta(N^{5}\log(N)^3)$ using more sophisticated polynomial and integer multiplication algorithms (see \cite{HH19}) and using a modular algorithm (see \cite[chapter 5]{GG13}), which we will describe shortly.

In order to avoid the bulk of these multiplications, we compute the values $Q_{n}(x,a)$ for fixed values of $x$ and $a$, rather than computing the individual coefficients. The nature of the equation means that $Q_{n}(x,a)$ cannot be determined directly from previous terms for $a=0$ or $x=0$, but it can be determined for any $a,x\neq 0$. Then, knowing that $Q_{n}(x,a)$ is a polynomial of degree at most $n+1$ in $a$ allows $Q_{n}(x,0)$ to be determined from $Q_{n}(x,1),Q_{n}(x,2),\ldots,Q_{n}(x,n+2)$. To be precise, these are related by the equation
\begin{equation}\label{eq:poly1}\sum_{j=0}^{n+2}(-1)^{j}{n+2\choose j}Q_{n}(x,j)=0.\end{equation}
Similarly, we can determine $Q_{n}(0,a)$ from $Q_{n}(1,a),Q_{n}(2,a),\ldots,Q_{n}(n+2,a)$ using the equation
\begin{equation}\label{eq:poly2}\sum_{j=0}^{n+2}(-1)^{j}{n+2\choose j}Q_{n}(j,a)=0.\end{equation}
Thus, the algorithm runs as follows:
\begin{algorithm}[caption={Computing $w_n$ for $n\leq N$.}, label={algo1}]
input: an integer $N$
output: a list output containing the numbers $w_n$ for $n\leq N$
for a from 0 to N+2
    for x from 0 to N+2
        $Q_{1}(x,a)$ $\gets$ $(x+1)^2(a+1)^2$
for n from 2 to N
    for x from 1 to N+2
        for a from 1 to N+2
            compute $Q_{n}(x,a)$ using $\eqref{eq:recursiveQ}$.
    for x from 1 to N+2
        compute $Q_{n}(x,0)$ using $\eqref{eq:poly1}$.
    for a from 0 to N+2
        compute $Q_{n}(0,a)$ using $\eqref{eq:poly2}$.
    Append $Q_{n}(0,0)$ to output.
\end{algorithm}

As equation \eqref{eq:recursiveQ} contains a sum in which $j$ runs from $1$ to $n-2$, this equation takes linear time to compute, so the number of operations in Algorithm~\ref{algo1} is of order $\Theta(n^{4})$. However, a nontrivial proportion of these operations are multiplication operations, which need to be taken into account in order to determine the time complexity of this algorithm. An efficient way to apply these multiplications is to convert this to a modular algorithm, which we describe in the next subsection. 
\subsection{Modular algorithms}
In order to avoid multiplying very large numbers, which can be computationally unwieldy, we implement a modular algorithm. That is, we choose a sequence of primes $p_{1},p_{2},\ldots, p_{k}$ whose product we know to exceed $w_n$ for $n\leq N$. Then, for each $j=1,\ldots, k$, we run the entirety of Algorithm \ref{algo1} with all operations being modulo $p_{j}$. For each prime, we then return the output values $w_n$ modulo that prime. Finally, the actual values $w_n$ can be recovered using the Chinese Remainder Theorem.

In theory, in the limit $N\to\infty$, the optimal time complexity (up to some constant) is achieved by setting $p_{j}$ to be the $j$th-smallest prime. This way, the number of primes required will be $\Theta(N/\log(N))$, while the size of the primes is typically $\Theta(\log(N))$, meaning that the multiplication operations typically take $\log(N)^2$ time (using naive multiplication). Since there are $\Theta(n^{4})$ such operations for each prime, the theoretical time complexity is $\Theta(N^{5}\log(N))$. Using a more efficient multiplication algorithm can reduce this theoretical time complexity to $\Theta(N^{5}\log(\log(N)))$.

\subsection{Implementation of algorithm and results}

In this section we describe how we implemented this algorithm to produce 1000 terms of the series $W(t)$. The program is given on the second author's website, as are the terms produced \cite{EPwebsite}.

As discussed in the previous section, the purpose of using a modular algorithm is to avoid time-intensive multiplications involving very large numbers. In practice, multiplications of (unsigned) 64-bit integers are carried out at the hardware level, so it is practical to insist that the primes be small enough that all multiplications can be applied in this way. That is, each prime $p_{j}$ should be less than $2^{32}$. While adhering to this restriction, we want to use the minimum possible number of primes, as the computation must be carried out for each one. This is achieved by setting each prime $p_{j}$ to be the $j$th-{\em largest} prime below $2^{32}$.

Recall that we require the product $P:=p_{1}p_{2}\cdots p_{k}$ to exceed $w_{N}$. Since we do not have good upper bounds on the numbers $w_{n}$ for large $n$, it is inefficient to choose the number of primes $k$ such that we can prove $P>w_{N}$ in advance. Instead we choose $k$ for which we simply {\em believe} that $P>N w_{N}$ as then we can apply the following proposition after the algorithm has run to prove that its output really is the desired sequence:
\begin{prop}\label{prop:algo_proof}
If the output sequence $\widetilde{w}_{1},\widetilde{w}_{2},\ldots,\widetilde{w}_{N}$ of the algorithm satisfies $N\widetilde{w}_{n}<P$ for each $n$, then  $\widetilde{w}_{n}=w_{n}$ for each $n$.
\end{prop}
\begin{proof}
By definition, we know that $\widetilde{w}_{n}$ is the remainder when $w_{n}$ is divided by $P$, so it suffices to show that each $w_{n}<P$.

We start with the fact that $w_{n}\leq nw_{n-1}$. This follows from the fact that removing the entry $1$ from a 3-stack-sortable permutation in $S_n$ and then reducing each remaining entry by $1$ yields a 3-stack-sortable permutation in $S_{n-1}$. It then follows from a simple induction on $n$ that $w_{n}<P$ and, therefore, $w_{n}=\widetilde{w}_{n}$.\end{proof}

We ran this algorithm with size $N=1000$ and with $k=105$ primes. The computation took 46 hours of computing, using 7.8GB of memory on a Dell Power Edge R820, 2.70GHz with 8 cores. The result of this computation was the sequence of numbers $\widetilde{w}_n$ for $1\leq n\leq 1000$, where $\widetilde{w}_n$ is the remainder when $w_{n}$ is divided by the product
\[P:=p_{1}p_{2}\cdots p_{105}\approx 2.9\cdot 10^{1011}.\]
The largest value of $\widetilde{w}_{n}$ was $\widetilde{w}_{1000}\approx 2.4\cdot 10^{975}$.
Hence, by Proposition \ref{prop:algo_proof}, the output is indeed the desired sequence $w_{1},w_{2},\ldots,w_{1000}$, which counts $3$-stack-sortable permutations.

\section{Initial ratio analysis}\label{Sec:Analysis}
In this and subsequent sections, we apply a number of numerical techniques to analyse the series so as to estimate the asymptotic behaviour of the coefficients. This, of course, is dominated by the growth constant $
\mu_3,$ but we are also interested in the sub-dominant terms, which are more difficult to estimate.

Let us first look at the ratios. If we have a simple power-law singularity, so that the generating function behaves as
\begin{equation} \label{eqn:F}
F(x) \sim C_1(1-\mu_3 x)^{-\theta},
\end{equation}
then with $[x^n]F(x) = c_n,$ it follows that the ratio of successive coefficients
\begin{equation} \label{rat}
r_n =\frac{c_n}{c_{n-1}} = \mu_3 \left ( 1 + \frac{g}{n} + o\left( \frac{1}{n} \right ) \right ),
\end{equation}
 with $g = \theta -1.$
A plot of the ratios of the 1000-term series against $1/n$ is shown in Figure~\ref{fig:rat}.
\begin{figure}[h!] 
   \centering
   \includegraphics[width=3in]{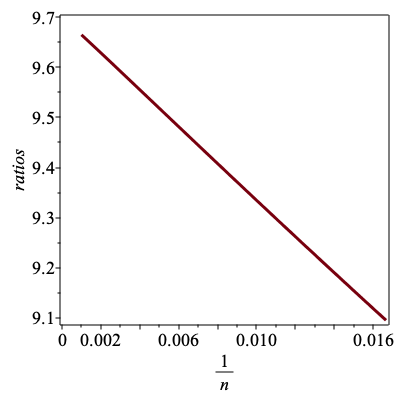} 
   \caption{Ratios vs. $1/n$ for $60 \le n \le 1000.$}
   \label{fig:rat}
\end{figure}

The plot appears to be linear and is going to a limit indistinguishable from $\mu_3=9.70$ at this level of precision. The gradient gives an estimate of the exponent $g,$ but not a very precise one.

Rather, if we know the value of $\mu_3,$ it follows from \eqref{rat} that we can estimate the exponent $g$
from estimators 
\begin{equation}\label{g}
g_n = \left (\frac {r_n}{\mu_3} - 1 \right ) n =g + o(1).
\end{equation}

Taking $\mu_3 = 9.70,$ we find for the estimators $g_n$ the values shown in Figure~\ref{fig:g}. These are plotted against $1/n,$ so we are interested in the point of intersection of the curve with the $y$-axis. It can be seen that the gradient of this plot changes sign at around $n=120,$ and there is a lot of curvature, so it is hard to estimate the limit as $n \to \infty$ with any precision. Nevertheless, one might guess a value around $-3.5$ or greater, which would imply that $\theta \approx -2.5$ in  \eqref{eqn:F}. But all one can really say with any confidence is that $g > -3.6.$ Note too that if we only had some 100 terms at our disposal we might erroneously conjecture that $g \approx -4.$

The curvature in Figure~\ref{fig:g} would suggest that the $o(1)$ correction term does not decay as $1/n,$ which would be the case for a simple power-law singularity, but rather decays more slowly. In the next section we give our reasons for believing that there are logarithmic factors in the generating function; we will give a more refined ratio analysis taking this into account in Section~\ref{sec:5.1}.

\begin{figure}[h!] 
   \centering
   \includegraphics[width=3in]{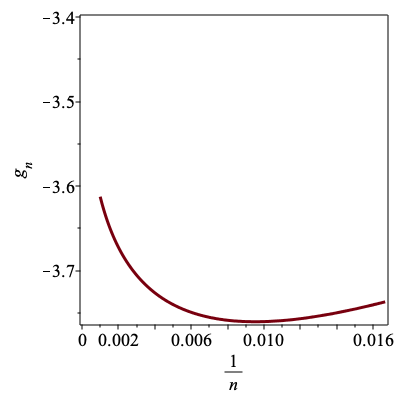} 
   \caption{$g$-estimates vs. $1/n$ for $60 \le n \le 1000.$}
   \label{fig:g}
\end{figure}

\section{Differential-approximant analysis}
Next, we performed a differential-approximant analysis. As described in \cite{G89},
the method of differential-approximants  fits the known coefficients of a power series to a number (typically 100 or more) of holonomic differential equations, and uses the critical parameters (the radius of convergence and exponent at that point) of those differential equations as estimators of the corresponding quantities of the underlying series expansion. 

More precisely, one uses the known series coefficients to find polynomials $Q_{k}(t)$ and $P(t)$ such that the power series solution $\widetilde{F}(t)$ of the  holonomic differential equation

\begin{equation} \label{eqn:da}
\sum_{k=0}^M Q_{k}(t)\left (t\frac{{\rm d}}{{\rm d}t}\right )^k \widetilde{F}(t) = P(t)
\end{equation}
agrees with the known coefficients of the function $F(t)$ being approximated. We refer to the order $M$ of this ODE as the {\em order} of the approximant.

Constructing such {\em differential-approximants} (DAs) is straightforward computationally, and it only involves solving a system of linear equations. Many such DAs are constructed by varying the degrees of the polynomials $Q_{k}(t)$ and $P(t)$ while still using most, or all, of the known series coefficients. The singularities are given by the zeros $z_i, \,\, i=1, \ldots , N_M$, of $Q_M(t),$ where $N_M$ is the degree of $Q_M(t).$   We take as the dominant singularity that which is both closest to the origin and common to all (or almost all) the DAs.
Critical exponents $\theta_i$ (\ref{eqn:F}) follow from the indicial equation of the DA. For the simplest (and most common) situation where there is a single root of $Q_M(t)$ at $z_i,$  
$$ \theta_i=M-1-\frac{Q_{M-1}(z_i)}{z_iQ_M'(z_i)}. $$ Slightly more complicated expressions are known for the cases of double, triple, etc. roots.
Further details of the method can be found in \cite{G89}.

 We used the 1000-term series and produced differential-approximants of 2nd to 8th-order. To get an idea of the rate of convergence, we analysed the 125-term series, then the 250-term series, then the 500-term series, and finally the full 1000-term series. Results are quoted with a level of precision representing the precision with which the estimates agreed among themselves. This is not to be taken as an estimate that is accurate to the number of quoted digits. Rather, one must consider the trends as the order of the underlying differential equation is increased, and also as the number of coefficients used is increased. Results were consistent across orders, but were slightly more consistent within orders as the order increased. That is to say, the standard deviation of estimates of the radius of convergence decreased slightly as the order of the approximants increased. 
 
 We show in Table~\ref{tab:xc} the estimates of the radius of convergence $x_c,$ and in Table~\ref{tab:exp} the corresponding exponent,
 which is $-\theta$ in the notation of equation \eqref{eqn:F}.
 
 \begin{table}
\begin{center}
\begin{tabular}{|l|l|l|l|l|}
\hline
Order & 125 terms & 250 terms & 500 terms & 1000 terms \\ \hline
2 & 0.10309660 & 0.1030966471 & 0.10309664855 & 0.1030966486134 \\
3 & 0.10309662 & 0.1030966481 & 0.10309664860 & 0.1030966486156 \\
4 & 0.10309663 & 0.1030966482 & 0.10309664861 & 0.1030966486158 \\
5 & 0.10309662 & 0.1030966482 & 0.10309664861 & 0.1030966486158 \\
6 & 0.10309662 & 0.1030966482 & 0.10309664861 & 0.1030966486158 \\
7 & 0.10309663 & 0.1030966484 & 0.10309664861 & 0.1030966486158 \\
8 & 0.10309663 & 0.1030966484 & 0.10309664861 & 0.1030966486158 \\
\hline
\end{tabular}
\caption{Estimates of radius of convergence, $x_c$} \label{tab:xc}
\end{center}
\end{table}

\begin{table}
\begin{center}
\begin{tabular}{|l|l|l|l|l|}
\hline
Order & 125 terms & 250 terms & 500 terms & 1000 terms \\ \hline
2 & 2.427 & 2.366& 2.314 & 2.276 \\
3 & 2.422 & 2.367 & 2.326 & 2.295 \\
4 & 2.442 & 2.368 & 2.328 & 2.282 \\
5 & 2.333 & 2.365 & 2.315 & 2.267 \\
6 & 2.415 & 2.360 &2.315 & 2.759 \\
7 & 2.415 & 2.363 & 2.315 & 2.250 \\
8 & 2.415 & 2.362 & 2.314 & 2.247 \\
\hline
\end{tabular}
\caption{Estimates of the exponent $-\theta.$} \label{tab:exp}
\end{center}
\end{table}

Using the 1000-term series and 4th order to 8th-order approximants, we estimated the radius of convergence as $x_c=0.103096648616(3).$
The exponent estimates are not stable and are decreasing with the number of terms. As we will see, this is characteristic of a generating function with a confluent logarithmic term, and we suggest that these exponent estimates will continue to decrease with increasing numbers of terms (approximately as $1/\log(n)$), eventually reaching a limit around 2.0.

The growth constant is therefore estimated as $\mu_3=1/x_c=9.69963634535(30),$ which is satisfyingly close to the ratio estimate above of $9.70$ but, of course, far more precise. The exponent, taken at face value, is also moderately consistent with the ratio analysis value of less than $-\theta \approx 2.5$ above.

However, there are other apparent singularities very close to $x_c.$ With a 300-term series, these are at $x=0.103096649223 \cdots,$  at $x=0.1030966541 \pm 4.77\times 10^{-9} i$, and at $x=0.1032969769 \cdots.$ With the 1000-term series, there are many more such close singularities. Having multiple singularities very close together is an indicator of a confluent singularity (i.e., not a pure power-law), and in particular of confluent logarithms.

A discussion of the behaviour of differential-approximants applied to the analysis of such a confluent singularity can be found in \cite{EG18}, where the analysis of the generating function of 4-valent planar Eulerian orientations is given. In that case, the generating function had a dominant singularity $E(x) = \frac{-x(1-\mu x)}{\log(1 - \mu x)},$ and the differential-approximants gave very precise estimates of the growth constant $\mu$ but suggested that the exponent was around $1.30.$ Additionally, there were other, less precisely located singularities whose location differed by 1 part in $10^{-7}$ from the true singularity. That is the hallmark of a confluent logarithmic singularity. (Loosely speaking, it seems to reflect the approximant attempting to prescribe a branch cut from $x_c$ to $\infty.$)

As another example, we constructed the test function $$F(x) = \frac{-x^3(1-x)^2}{\log^3(1-x)} \left ( 1 + e^x \right ).$$ Using 3rd-order differential-approximants and 150 terms of the series, we found the dominant singularity to be at $x_c \approx 0.999999528\cdots$ with an exponent of $2.46217.$ There are other apparent singularities at $x=1.00000395\cdots$ and at $x=1.0221\cdots .$
Note that in both these cases the exponent is wrongly estimated by the differential-approximants, in each case being some 20-25\% greater than the true value.

We suggest that the same phenomenon is happening with the generating function for $3$-stack-sortable permutations and that the correct exponent value is more likely to be around $2.0$, with a confluent logarithmic term. Of course, we are not claiming to be able to distinguish between an exponent value of, say, 1.9 with a confluent logarithm, or $2.1$ and a confluent logarithm, or power of a logarithm, but in our experience confluent logarithms are usually associated with integer or (rarely) half-integer exponents. We tentatively suggest that the exponent is exactly $2.0$, and provide further evidence for this below.

The conclusion from this analysis is the conjecture that the generating function behaves as
\begin{equation} \label{W}
W(t) \sim C_0(1-\mu_3 t)^\alpha \cdot \log^\beta(1-\mu_3 t),
\end{equation}
with $\alpha \approx 2$. As we show below, if $\alpha=2$ exactly, our best estimate of $\beta$ is $-2$, but with considerable uncertainty. Consequently,
 \begin{equation} \label{Wcoeff}
 [t^n]W(t)=w_n\sim \frac{c_0\mu_3^n}{  n^{(\alpha+1)}\cdot \log^\lambda{n}} ,
 \end{equation}
where $\mu_3 = 9.69963634535(30),$ $\alpha \approx 2,$ and $\lambda = -\beta$ if $\alpha$ is not a positive integer, and $\lambda=-\beta+1$ if $\alpha$ is a positive integer. In the next section, we provide additional evidence for this conjectured form. 

\section{Series extension}
The conjectured presence of logarithms greatly complicates the analysis. Whenever logarithms are involved, one really needs very long series to extract believable asymptotics from ratio-based analysis.
In an attempt to get further terms, we use the method of series extension to obtain many more {\em approximate} coefficients. In this way, we have extended the series by 1000 approximate terms, using differential-approximants to estimate the additional coefficients.  The details of this method are given in \cite{G16}. These 1000 approximate terms are given on the second author's website \cite{EPwebsite}.

The basic idea is to use families of differential approximants to predict higher order terms. The rationale here is that every differential approximant, being a holonomic differential equation with a power series solution, generates, in principle, all coefficients. Of course, unless one has stumbled upon the exact solution, which sometimes happens, all coefficients beyond the order used to obtain the differential approximant will be approximate. 

In extending the series, we construct many (typically 100 or more) differential approximants obtained by varying the degrees of the polynomials multiplying the various derivatives (the $Q_k(t)$ and $P(t)$ in equation~\eqref{eqn:da}) and taking the mean of these coefficients (discarding outliers). We then use the standard deviation to estimate the accuracy of the coefficients. As the order of the coefficients increases, the calculated error also increases, and we stop when the error (taken to be 1 standard deviation), exceeds some desired value.  

As we propose to use these coefficients in a ratio type analysis, even 5 or 6 digit precision is often sufficient, as the error in the ratio is visually imperceptible. In this case we obtained 1000 additional, approximate terms that we believe to be accurate, at worst, to 20 significant digits. 

 Initial reaction to this method is often one of disbelief. As a {\em proof of concept}, we first demonstrate the method on the known series for two stacks in parallel (tsip) \cite{EPG17}, known to 501 terms. Mirroring our initial knowledge of the 3-stack-sortable series, which was first extended by the first author to 174 terms \cite{D20}, we take the first 174 coefficients of the tsip series and predict the next 327 coefficients. We use 6th-order differential approximants. The correct value of the coefficient of $t^{501}$ is  $1.947305150994482942849937863820882187009\ldots \times 10^{451}.$ The predicted value is $1.9473051509944829428499378638083404 \times 10^{451},$ and the predicted accuracy is 29 digits, which is seen to be the case. Lower order terms are predicted to greater accuracy. Similarly, originally given the 174 term series for 3-stack-sortable permutations, we generated a further 530 approximate terms. When we extended the series, initially to 300 terms, we  confirmed that our estimate of the coefficient of $t^{300}$ was correct to 29 significant digits, as expected.

\subsection{Analysis of extended series}\label{sec:5.1}
We have repeated the ratio analysis with this extended series, now with 2000 terms,  and everything remains consistent with the results of the analysis of the exact series. But we can now extract more precise information and start looking for the effect of confluent logarithmic terms.

A plot of the ratios using the estimated terms is shown in Figure~\ref{fig:rat2000}. This should be visually indistinguishable from the corresponding plot with the exact (but unknown) coefficients, as the ratios are expected to be accurate to more than 10 significant digits.
\begin{figure}[h!] 
   \centering
   \includegraphics[width=3in]{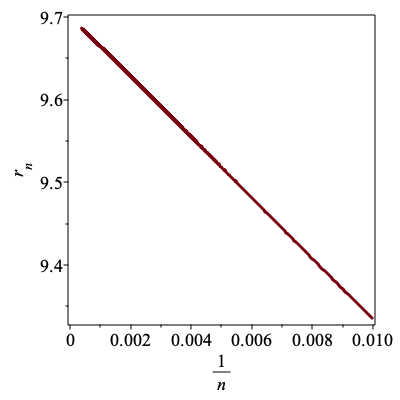} 
   \caption{Ratios vs. $1/n$ for $100 \le n \le 2000.$}
   \label{fig:rat2000}
\end{figure}

The plot is still linear and is still going to a limit indistinguishable from $\mu_3=9.70$ at this level of precision. However, we can do considerably better than this with the extended series.
First, we show in Figure~\ref{fig:rl} a plot of the linear intercepts of the ratios 
\begin{equation}\label{eq:li}
\ell_n = n\cdot r_n - (n-1)\cdot r_{n-1} \sim \mu_3\left ( 1 + o(1) \right ),
\end{equation}
 where $r_n$ is just the ratio of the coefficients. Constructing the linear intercepts eliminates the dominant $O(1/n)$ term in the ratio plot (see equation \eqref{rat}). From this figure, it is clear that one needs some 300 terms before the fact that the plot gradient changes sign becomes apparent. This plot reaches its maximum at $n=202$ and appears to be going to a limit below 9.70, and indeed rather to 9.6996, consistent with the (much more precise) differential-approximant analysis.
The extra terms of the extended series makes this behaviour manifest. With only 174 terms, one might well conjecture that the limit was slightly greater than 9.70, as the first author did in \cite{D20}.

\begin{figure}[h!] 
   \centering
   \includegraphics[width=3in]{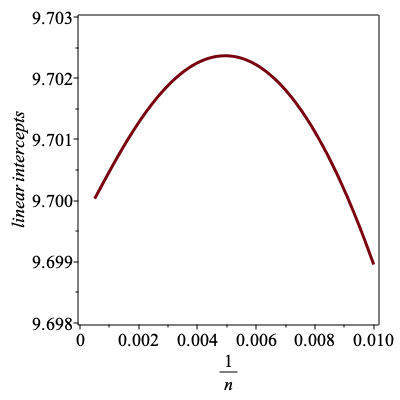} 
   \caption{Linear intercepts, which approach $\mu_3$ vs. $1/n$ for $100 \le n \le 2000.$}
   \label{fig:rl}
\end{figure}

We repeated the simple analysis of exponent estimators $g_n$ with the extended series. The results are shown in Figure~\ref{fig:gext}, which should be compared to the Figure~\ref{fig:g}, produced using only the exact series coefficients. From the extended series, one concludes that $g > -3.5,$ but estimating the intercept with the $y$-axis accurately is not really possible.

\begin{figure}[h!] 
   \centering
   \includegraphics[width=3in]{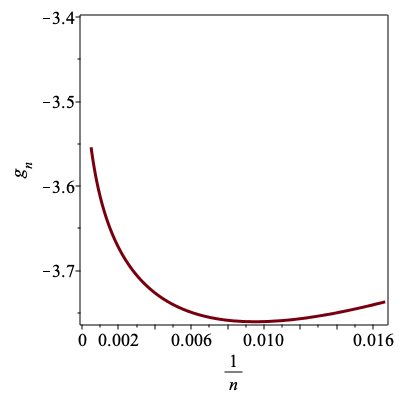} 
   \caption{ $g$-estimates vs. $1/n$  for $50 \le n \le 2000.$}
   \label{fig:gext}
\end{figure}

Given the evidence from the differential-approximant analysis that confluent logarithms are likely present, we can say more about the asymptotics, and we can also better understand the poor convergence of the $g_n$ estimators.

 From Flajolet and Sedgewick \cite[page 385]{FS09},  we see that if
\begin{equation} \label{eqn:general}
f(x) = (1-\mu \cdot x)^{\alpha}\left ( \frac{1}{\mu \cdot x} \log \frac{1}{1-\mu \cdot x} \right )^\beta,
\end{equation}
 then 
\begin{equation}\label{eqn:coeff}
f_n=[x^n]f(x) \sim \frac{ \mu^n \cdot n^{-\alpha-1}}{\Gamma(-\alpha)}(\log{n})^\beta \left(1+\frac{c_1}{\log{n}}+\frac{c_2}{\log^2{n}} +\frac{c_3}{\log^3{n}} + \frac{c_4}{\log^4{n}} + \cdots \right ),
\end{equation}
 where $$c_k={\beta \choose k}\Gamma{(-\alpha)} \frac{d^k}{ds^k}\frac{1}{\Gamma(s)}\bigg\rvert_{s=-\alpha}.$$ When $\alpha$ is a positive integer, the evaluation of the constants must be interpreted as a limiting case as the $\Gamma$ function diverges, so that certain constants vanish. In particular, provided that $\alpha$ is a positive integer (which we suggested based on the differential-approximant analysis) and $\beta$ is not a positive integer, one has
$$[x^n]f(x) \sim  \mu^n \cdot n^{-\alpha-1}(\log{n})^\beta \left(\frac{c_1}{\log{n}}+\frac{c_2}{\log^2{n}} +\frac{c_3}{\log^3{n}} + \frac{c_4}{\log^4{n}} + \cdots \right ).$$
The ratio of successive coefficients is, in the general case (by which we mean $\alpha$ not a positive integer),
\begin{equation}
r_n = \frac{[x^n]f(x)}{[x^{n-1}]f(x)} \sim \mu \left (1 - \frac{\alpha+1}{n} + \frac{\beta}{n\log{n}}+ \frac{c_1}{n\log^2{n}} +  \cdots  \right ),
\end{equation}  but in the case that $\alpha$ is a positive integer and $\beta$ is not a positive integer, one has 
\begin{equation}\label{eqn:lograts}
r_n = \frac{[x^n]f(x)}{[x^{n-1}]f(x)} \sim \mu \left (1 - \frac{\alpha+1}{n} + \frac{\beta-1}{n\log{n}}+ \frac{c_1}{n\log^2{n}} +  \cdots  \right ). 
\end{equation}

One can estimate  $\alpha$ from the sequence $$g_n = \left ( \frac{r_n}{\mu} -1 \right )\cdot n  \sim -\alpha-1 +\frac{\beta}{\log{n}}+\frac{c_1}{\log^2{n}} + \cdots.$$ This is the same exponent estimator plotted in  Figure~\ref{fig:gext}, but this refined analysis shows that $g_n$ should be plotted not against $1/n$ but rather against $1/\log{n},$ as we do in Figure~\ref{fig:gext-log}. Extrapolation to a value around $-3.0$ seems much more convincing from this plot.

Furthermore, the gradient of this plot should give a measure of $\beta$ or $\beta-1$ if, as it seems, $\alpha$ is a positive integer. Unfortunately, the curvature of the plot means it is not possible to accurately estimate the limiting gradient.

We can try and  estimate the exponent $\beta$ by setting $\alpha$ to exactly $2$, and so estimates of $\beta-1$ are given by
\begin{equation}\label{eqn:fit-alpha}
\beta_n-1 = \left ( \left ( \frac{r_n}{\mu} -1 \right )\cdot n+\alpha+1\right )\log{n}  \sim \beta-1+\frac{c_1}{\log{n}} + \cdots,
\end{equation}
 where the numerical estimates of $\mu$ and $\alpha$ are used. In this way, we obtain the results shown in Figure~\ref{fig:beta1}, which shows the difficulties of estimating the power of confluent logarithmic terms. The plot is clearly displaying a minimum, but it is difficult to extrapolate to $n \to \infty.$ If we assume integer values for the exponent, $\beta-1 = -3$ appears likely, but $\beta-1 = -2$ appears possible.
 
 Another way of estimating both $\mu_3$ and $\beta$ is to look more closely at the linear intercepts $\ell_n$ defined in \eqref{eq:li}. We can give more terms in the asymptotic expansion of $\ell_n$ by making use of \eqref{eqn:lograts}. Taking these terms into account, we can rewrite \eqref{eq:li} as
 \begin{equation}\label{eq:li2}
\ell_n = n\cdot r_n - (n-1)\cdot r_{n-1} \sim \mu_3\left ( 1 + \frac{(1-\beta)}{n\log^2{n}}+ \frac{c_1}{n\log^3{n}}+ \cdots \right ),
\end{equation}
assuming $\alpha$ is a positive integer; otherwise, $(1-\beta)$ should be replaced by $-\beta.$ This shows that $\ell_n$ should be plotted against $1/(n\log^2{n})$ and not $1/n$ as was done in Figure~\ref{fig:rl}. We show in Figure~\ref{fig:li2} the linear intercepts plotted in this way. It can be seen that the extrapolated limit is around 9.6996, as before. However, we can also use this data to estimate $\beta$ by forming from equation \eqref{eq:li} the estimators
\begin{equation}\label{eq:li2}
 \beta_n-1 \equiv \left (1-\frac{\ell_n}{\mu_3} \right )\cdot n \log^2{n} \sim \beta-1 + \frac{2c_1}{\log{n}} + \cdots.
\end{equation}
We show this data in Figure~\ref{fig:beta2}, plotted against $1/\log{n},$ as appropriate. It is clear that this is impossible to extrapolate, as it appears to be forming a minimum. One would need data minimally to $1/\log{n}=0.1$ and ideally to $1/\log{n}=0.05$ to confidently extrapolate this plot. This corresponds to 22,000 and 485 million terms respectively, which is of course out of the question. This highlights the difficulty of analysing for powers of logarithms.

 \begin{figure}[ht]
\setlength{\captionindent}{0pt}
\begin{minipage}{0.48\textwidth}
  \includegraphics[width=0.97\linewidth]{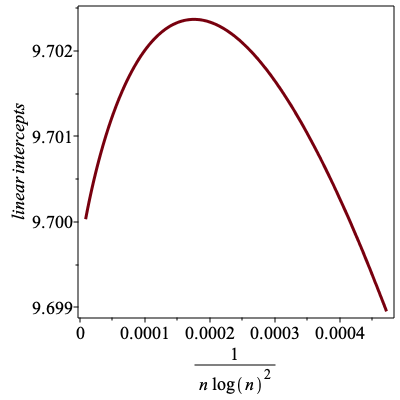} 
   \caption{Plot of linear intercepts  vs. $1/{n\log^2{n}},$ for $100 \le n \le 2000.$ }
   \label{fig:li2}
\end{minipage}\hfill
\begin{minipage}{0.48\textwidth}
   \includegraphics[width=0.97\linewidth]{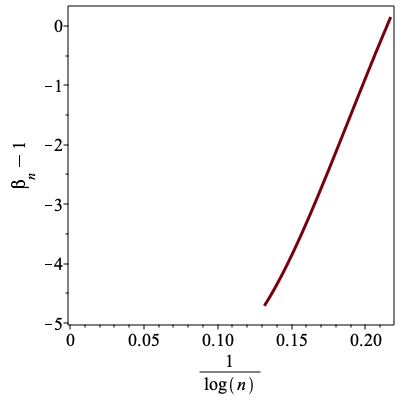} 
   \caption{Plot of $t_3=\beta-1$  vs. $1/{n},$ for $100 \le n \le  2000.$}
   \label{fig:beta2}
\end{minipage}
\end{figure}

Our tentative conclusion from this analysis is that the generating function for 3-stack-sortable permutations behaves as 
\begin{equation}\label{eqn:c}
W(t) \sim C\cdot (1-\mu_3 \cdot t)^{\alpha}\cdot \left ( \frac{1}{ t} \log \frac{1}{1-\mu_3 \cdot t} \right )^{\beta},
\end{equation}
 where $\alpha \approx 2,$ and $\beta \in [-3,-1]$ if $\alpha=2.$ If $\alpha \ne 2,$ it is too difficult to estimate $\beta,$ as it is highly sensitive to the estimate of $\alpha.$

 \begin{figure}[h!] 
   \centering
   \includegraphics[width=3in]{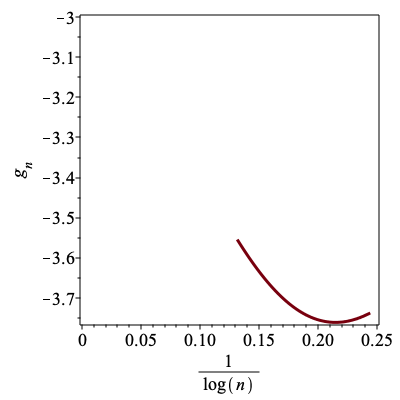} 
   \caption{ Estimates of $g$ vs. $1/\log{n}$  for $50 \le n \le 2000.$}
   \label{fig:gext-log}
\end{figure}

 \begin{figure}[h!] 
   \centering
   \includegraphics[width=3in]{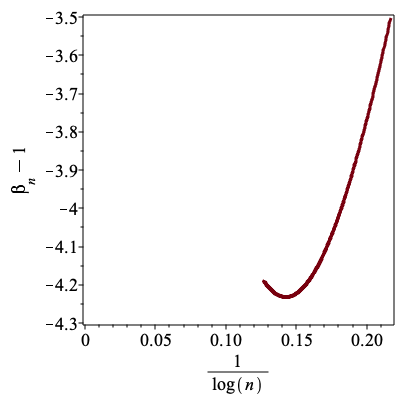} 
   \caption{ Estimates of $\beta-1$ vs. ${1/\log{n}}$  for $100 \le n \le 2000,$ assuming $\alpha=2.$}
   \label{fig:beta1}
\end{figure}
  To make these conclusions more convincing, we will re-estimate the asymptotics in two distinct ways: first from the coefficients directly, and second from the ratios as given in \eqref{eqn:lograts}.
  
  From equation \eqref{eqn:coeff}, we have\footnote{If $\alpha$ is a positive integer, $\Gamma(-\alpha)$ diverges. This is taken care of by replacing the pre-multiplier by some constant $C,$ as in \eqref{eqn:c}.}
  \begin{equation} \label{eqn:logcoeff}
  \log{f_n} = -\log{\Gamma(-\alpha)} +n \log{\mu_3} - (\alpha+1)\log{n}+(\beta-1)\log(\log{n}) +\frac{d_1}{\log{n}} +\frac{d_2}{\log^2{n}} + \cdots,
  \end{equation}
  assuming $\alpha$ is a positive integer, otherwise replace $\beta-1$ by $\beta.$
  
  We will use our estimate of $\mu_3$ and then fit to $$d_n\equiv \log{f_n}-n\log{\mu_3}=t_1+t_2\log{n}+t_3 \log(\log{n}) + \frac{t_4}{\log{n}} +\frac{t_5}{\log^2{n}}$$ by taking successive quintuples of coefficients $\{d_{k-2},\, d_{k-1},\, d_k,\, d_{k+1},\, d_{k+2}\},$ letting $k$ range over all values up to 2000, reflecting the length of the extended series at our disposal. The principal parameters of interest are $t_2=-\alpha-1$ and $t_3=\beta-1$ (assuming $\alpha$ is a positive integer), and plots of these quantities estimated in this way are given in Figures~\ref{fig:t2} and \ref{fig:t3} respectively.
 
  \begin{figure}[h!]
\setlength{\captionindent}{0pt}
\begin{minipage}{0.48\textwidth}
  \includegraphics[width=0.97\linewidth]{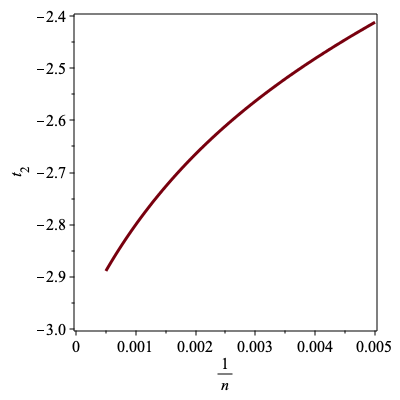} 
   \caption{Plot of $t_2=-\alpha-1$  vs. $1/{n},$ for $200 \le n \le 2000.$ }
   \label{fig:t2}
\end{minipage}\hfill
\begin{minipage}{0.48\textwidth}
   \includegraphics[width=0.97\linewidth]{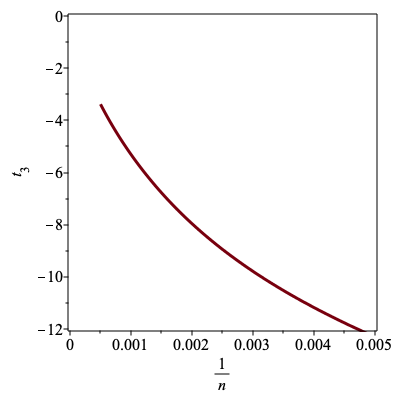} 
   \caption{Plot of $t_3=\beta-1$  vs. $1/{\log(n)},$ for $200 \le n \le  2000.$}
   \label{fig:t3}
\end{minipage}
\end{figure}

From Figure~\ref{fig:t2}, one sees $t_2=-\alpha-1$ convincingly approaching our previously-estimated value of $-3$ exactly, but the plot in Figure~\ref{fig:t3} cannot be reliably visually extrapolated. Indeed, there is the hint of a turning point, which could only be confirmed if we had still further terms. Investigating this, we tried fitting to fewer or more terms in the asymptotics. That is to say, here we have fitted to five unknown parameters. We tried fitting to just three parameters, then four, then six. We found that the estimates of $t_2$ were robust, all coming in around $-3,$ but the estimates of $t_3$ were not at all robust. We conclude that this is not a good method for estimating the exponent $\beta$ with the number of terms at our disposal, but that the earlier estimate $\alpha\approx 2$ is well-supported.
  
  A second method is to perform a similar fitting procedure with the sequence of ratios. Let $r_n = c_n/c_{n-1}$ be the ratio of successive coefficients in the generating function. Then
  $$e_n \equiv \left ( \frac{r_n}{\mu_3} -1 \right )\cdot n  = t_1 +\frac{t_2}{\log{n}}+\frac{t_3}{\log^2{n}}  +\frac{t_4}{\log^3{n}} .$$ As above, we assumed the estimated value of the growth constant $\mu_3$ and fitted successive quartets of coefficients $\{d_{k-2},\, d_{k-1},\, d_k,\, d_{k+1}\},$ letting $k$ range over all values up to 2000, reflecting the 2000-term (approximate) ratio series at our disposal.  The principal parameters of interest are $t_1=-\alpha-1$ and $t_2=\beta-1$, and plots of these quantities estimated in this way are given in Figures~\ref{fig:t1r} and \ref{fig:t2r} respectively. The plots are again consistent with our previous estimate, $\alpha \approx 2,$ while the plots of estimators of $\beta$ appear to be approaching a maximum, making extrapolation impossible.
  
  We repeated this analysis, fixing the value of $\alpha$ at 2 exactly, in the hope that it would give a more precise estimate of $\beta.$  More precisely, in equation \eqref{eqn:logcoeff}, we set $\alpha = 2$. Then fitting as described, we obtain the results shown in Figure~\ref{fig:beta3}. This is consistent with our previous tentative conclusion, $\beta-1=-3.$
  
  \begin{figure}[h!] 
   \centering
   \includegraphics[width=3in]{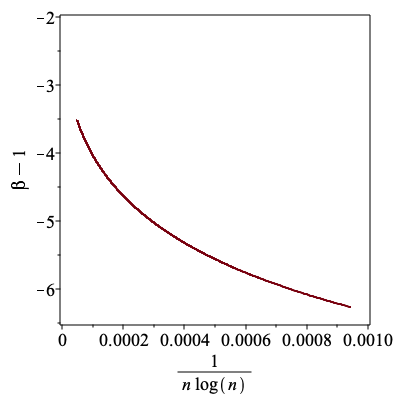} 
   \caption{ Estimates of $\beta-1$ vs. $1/{(n\log{n})}$  for $100 \le n \le 2000,$ assuming $\alpha=2.$}
   \label{fig:beta3}
\end{figure}

In yet another attempt to estimate $\beta,$ we first assumed that $\alpha=2.$ Then we normalised the coefficients by defining new coefficients $s_n \equiv w_n\cdot n^3/\mu^n.$ We expect $s_n \sim c_0\cdot \log^{\beta-1} n,$ and the ratios
$$R_n \equiv s_n/s_{n-1} \sim 1+\frac{\beta-1}{n\log{n}}.$$ Then $(R_n-1)\cdot n\log{n} \sim \beta-1.$ 

In Figure~\ref{fig:beta4}, we plot this estimator of $\beta-1.$
Note the turning point at around $n=700.$ This clearly demonstrates the importance of having sufficient terms if one is to have any chance of estimating the power of logarithmic terms. This plot illustrates that a limit of $-4$ appears very unlikely, while $-3$ appears plausible, but we certainly could not rule out $-2$. Moreover, this discussion ignores the possibility of non-integer values of $\beta.$
  
  \begin{figure}[h!] 
   \centering
   \includegraphics[width=3in]{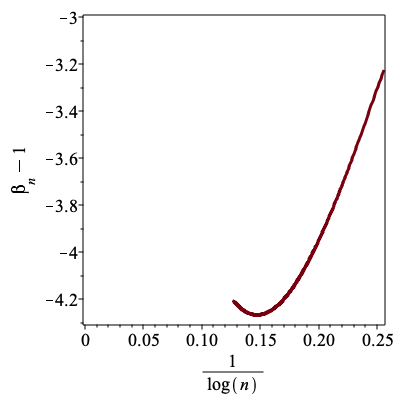} 
   \caption{ Estimates of $\beta-1$ vs. $1/{(\log{n})}$ from normalised ratios $R_n$  for $100 \le n \le 2000,$ assuming $\alpha=2.$}
   \label{fig:beta4}
\end{figure}

Another attempt to estimate the value of $\beta$, under the assumptions that $\beta$ is an integer and that $\alpha=2$, is the following: The asymptotic behaviour of the normalised coefficients $s_n$ is, from (\ref{eqn:coeff}), $$s_n \sim \frac{e_1}{n^{1-\beta}} + \frac{e_2}{n^{2-\beta}}+\frac{e_3}{n^{3-\beta}}+ \cdots .$$ We can assume the value of $\beta$ and, as discussed above, fit successive triples of coefficients $\{s_{k-1},\, s_{k},\,  s_{k+1}\}$ to estimate the amplitudes $e_i,\,\, i=1,2,3.$ If we choose the correct value of $\beta,$ these estimates should converge. If we choose the wrong value, the estimates should either diverge or vanish. Choosing $\beta=-1,$ the leading amplitude $e_1$ is estimated to be zero, implying the absence of a term O$(n^{-2}).$ Choosing $\beta=-3,$ the leading amplitude $e_1$ appears to be diverging, as do estimates of $e_2$ and $e_3.$ With $\beta=-2,$ estimates of $e_1$ appear to be converging to a value around 1.0. This is our strongest evidence that $\beta \approx -2.$
  
  We conclude this section by asserting that our best estimate of $\alpha$ is $\alpha=2.$ However, looking at a wide variety of fits as described above, it would be more prudent to give the estimate $\alpha = 2.0 \pm 0.25.$ If $\alpha=2,$ we hesitantly offer $\beta=-2$ as our best estimate, but with a wide uncertainty of $\pm 1$. If $\alpha \ne 2,$ we give no estimate of $\beta.$

\begin{figure}[ht]
\setlength{\captionindent}{0pt}
\begin{minipage}{0.48\textwidth}
  \includegraphics[width=0.97\linewidth]{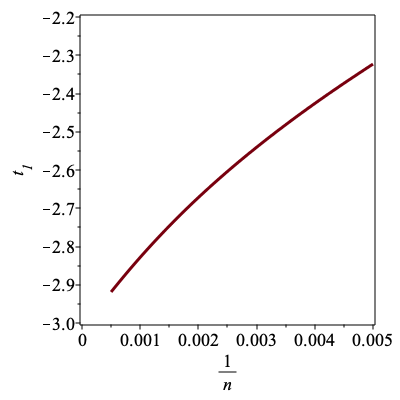} 
   \caption{Plot of $t_1=-\alpha-1$  vs. $1/{n},$ from ratios for $200 \le n \le 2000.$ }
   \label{fig:t1r}
\end{minipage}\hfill
\begin{minipage}{0.48\textwidth}
  \includegraphics[width=0.97\linewidth]{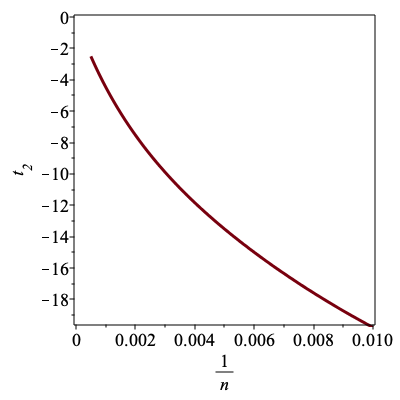} 
   \caption{Plot of $t_2=\beta-1$  vs. $1/{\log(n)},$ from ratios for $200 \le n \le 2000.$}
   \label{fig:t2r}
\end{minipage}
\end{figure}

\subsection{Bounds}

From the 174 exact coefficients he obtained, the first author gave the rigorous lower bound $\mu_3 \ge w_{174}^{1/174} = 8.659702\ldots$ \cite{D20}. From our 1000-term series, we have the improved result 
$$\mu_3 \ge  w_{1000}^{1/1000} = 9.44879\ldots .$$
The observation used to prove this inequality was that the sum $\lambda\oplus\pi$ of two 3-stack-sortable permutations $\lambda$ and $\pi$ is also 3-stack-sortable. Using this fact allows us to deduce a stronger bound. The following method was used in a similar context in \cite[Section 4]{EP19}, as well as a variety of different contexts where objects are decomposable (see, for example, \cite[p. 89--92]{MS93}). The idea is that we know that $W(t)$ can be written as
\[W(t)=\frac{\widetilde{W}(t)}{1-\widetilde{W}(t)},\]
where $\widetilde{W}(t)$ is the generating function for primitive (also called sum-indecomposable) 3-stack-sortable permutations. The 1000 coefficients of $W(t)$ yield 1000 coefficients of $\widetilde{W}(t)$. Moreover, by the equation relating $W(t)$ and $\widetilde{W}(t)$, we must have $\widetilde{W}(t)<1$ for $0<t<1/\mu_3$. We find that the first $1000$ terms of the series $\widetilde{W}(t)$ add to $1$ for $t=t_c=0.105424\ldots$, so
\[\mu_3 \ge 1/t_c =9.4854\ldots\]
Based on our approximate terms, we expect that this bound would improve to 
\[\mu_3 \ge  9.5828\ldots\]
given 2000 exact terms. Let us also remark that the best known rigorous upper bound for $\mu_3$, obtained in \cite{DefantPreimages}, is $12.53296$. 

In \cite{Bona, BonaSurvey}, B\'ona conjectured that $|\mathcal W_k(n)|\leq\binom{(k+1)n}{n}$ for all $n,k\geq 1$. In \cite{D20}, it was suggested that this conjecture is likely false, and it was shown that it contradicts a separate conjecture of B\'ona's stating that the sequence $(w_n)_{n\geq 1}$ is log-convex (meaning that the ratios $w_{n+1}/w_n$ are increasing). The first author was not able to fully disprove B\'ona's first conjecture because he did not have a sufficient lower bound for $\mu_3$. The rigorous lower bound $\mu_3\geq 9.4854\ldots$ that we have just obtained is significant because it allows us to disprove B\'ona's first conjecture.  Indeed, for $k=3$, the inequality $w_n\leq\binom{4n}{n}$ would imply that $\mu_3\leq 256/27\approx 9.481$, which we now know is not the case. On the other hand, our new data shows that the first $1000$ terms in the sequence $(w_n)_{n\geq 1}$ are log-convex, lending even more evidence toward B\'ona's second conjecture.

\section{Conclusion}
We extended the previous 174-term series that counts 3-stack-sortable permutations to 1000 terms and analysed the series. From this series, we conjecture that the generating function behaves as
$$W(t) \sim  C\cdot (1-\mu_3 t)^\alpha \cdot \log(1 - \mu_3 t)^\beta. $$
so that $$[t^n]W(t)=w_n\sim c_0\mu_3^n\cdot n^{-\alpha-1}(\log{n})^{\beta-1} ,$$
where $\mu_3 =  9.69963634535(30)$, $\alpha=2.0 \pm 0.25,$ and if $\alpha=2$ then $\beta = -2 \pm 1.$ 

The estimate of the growth constant (just) contradicts the conjecture of the first author that $$9.702 < \mu_3 \le 9.704.$$ 
We extended the series, using differential-approximants, to obtain an additional $1000$ approximate coefficients and used the approximate coefficients and ratio-analysis methods to confirm the results obtained from the differential-approximant analysis of the exact coefficients. We are confident of our estimate of the growth constant. We are slightly less confident of the estimate of the value of the exponent $\alpha$ and much less confident of our estimate of the value of $\beta.$ 

We have also attempted to identify $\mu_3$ experimentally as an algebraic number, and as a product of fractional powers of small primes, but without success.

\end{document}